\documentclass[12pt,widepage]{article}
\usepackage[cp1251]{inputenc}
\usepackage[english]{babel}

\usepackage{amsmath,amsthm,amssymb}
\usepackage{hyperref}
\usepackage{enumerate}
\newtheorem{theorem}{\noindent Theorem}
\newtheorem{lemma}{\noindent Lemma}
\newtheorem{corollary}{\noindent Corollary}
\newtheorem{defe}{\noindent Definition}
\newtheorem{Th}{Theorem}
\newtheorem{Lm}[Th]{Lemma}
\newtheorem{Co}[Th]{Corollary}
\newtheorem{proposition}[Th]{Proposition}

\newtheorem{remark}[Th]{Remark}

\def\eps{\varepsilon}

\DeclareMathOperator{\thi}{th}
\DeclareMathOperator{\sthi}{sth}
\DeclareMathOperator{\supp}{supp}
\newcommand{\me}{\rm{qbs}}
\newcommand{\qbs}{\rm{qb}}
\newcommand{\thn}[1]{\|#1\|_{_{\mathrm{SR}^{1}}}}
\newcommand{\subsetmod}{\vspace{1pt}{\scriptscriptstyle \stackrel{\mathrm{mod 0}}{\subset}}}
\newcommand{\geqmod}{\vspace{1pt}{\scriptscriptstyle \stackrel{\mathrm{mod 0}}{\geq}}}
\newcommand{\leqmod}{\vspace{1pt}{\scriptscriptstyle \stackrel{\mathrm{mod 0}}{\leq}}}

\date{}
\author{A.~M.~Vershik$^{a,b,c}$, P.~B.~Zatitskiy$^{a,d}$, F.~V.~Petrov$^{a,b}$}

\title{VIRTUAL CONTINUITY OF MEASURABLE FUNCTIONS AND ITS APPLICATIONS}

\begin{document}

\maketitle

\abstract{Classical theorem of Luzin states that a measurable function of one real
variable is ``almost'' continuous. For measurable functions of several variables
the analogous statement (continuity on the product of sets having almost full
measure) does not hold in general. Searching for a right analogue of Luzin theorem
leads to a notion of virtually continuous functions of several variables. This
probably new notion implicitly appears in the statements of embedding theorems
and trace theorems for Sobolev spaces. In fact it reveals the nature of
such theorems as statements about virtual continuity.
Our results imply that under conditions
of Sobolev theorems there is a well-defined integration of
a function over wide class of singular measures, including the measures concentrated
on submanifolds. The notion of virtual continuity is used also
for the classification of measurable functions of several variables and in
some questions on dynamical systems, theory of polymorphisms and bistochastic
measures. In this paper we recall necessary definitions and properties of
admissible metrics, give several definitions of virtual  continuity
and discuss some applications. Revised version (without the proofs) is published
in \cite{VZPFA}.}

\let\thefootnote\relax\footnote{
a) St. Petersburg Department of V.~A.~Steklov Institute of Mathematics RAS;
b) St. Petersburg State University;
c) Kharkevich Institute for Information Transmission Problems RAS;
d) P.~L.~Chebyshev Laboratory in St. Petersburg University.
E-mail: avershik@gmail.com, paxa239@yandex.ru, fedyapetrov@gmail.com.
Our work is supported by RFBR grants 14-01-00373A and 13-01-12422-OFI-m;
President of Russia grant MK-6133.2013.1; Government of Russia (Chebyshev Laboratory)
grant
11.G34.31.0026; JSC ``Gazprom Neft''; SPbSU grant (project 6.38.223.2014).}

\newcommand{\thefootnote}{\arabic{footnote}}

\section{Introduction. Admissible metrics}


We consider a standard Lebesgue--Rokhlin probabilistic space with continuous (atomless) measure, isomorphic to the segment $[0,1]$ with Lebesgue measure.
The first author  \cite{V1,V2,V3} suggested to consider on a fixed standard measure
space $(X, {\mathfrak A}, \mu)$ different ({\it admissible}) metrics,
in the contrary to usual approach, when a metric space is fixed and Borel
measures vary. Such an approach is useful and necessary in ergodic theory and
other situations. Matching the metric and measure structures leads to the notion
of a metric (admissible) triple:

\begin{defe}
Metric \textup(or semimetric\textup) $\rho$\textup,
which is measurable as a function of two variables on a standard
space $(X, {\mathfrak A}, \mu),$ is called
\emph{admissible}\textup, if there exists a measurable subset $X_0\subset X$
of full measure\textup, $\mu(X_0)=1$\textup, so that\textup a metric
\textup(resp. semimetric\textup) space $(X_0, \rho)$ is separable.

A standard measure space $(X,\mu)$ equipped by an admissible
\textup(semi\textup)metric $\rho$ is called an admissible metric
triple or just {\it admissible triple} $(X,\mu,\rho)$.
\end{defe}

Properties of admissible semimetrics and metrics are studied in details in our previous
papers \cite{ZP}, \cite{VPZ}. In particular, a number of equivalent definitions
of admissibility is given.

\begin{proposition} If $\rho$ is an admissible metric on $(X, {\mathfrak A}, \mu)$\textup,
then completed Borel sigma-algebra  ${\mathfrak B}={\mathfrak B}(X,\rho)$ is
a subalgebra of  ${\mathfrak A}$\textup and measure $\mu$ is inner regular w.r.t the
metric $\rho$\textup, i.e. for any $A \in {\mathfrak A}$ we have
$$
\mu(A)=\sup\{\mu(K)\colon K\subset A,  K \mbox{ is compact
in a metric} \rho\}.
$$
Thus for any admissible metric initial
measure $\mu$ is a Radon measure in $(X,\rho)$.
\end{proposition}

\begin{proof}
Measurability of Borel sets, i.e. inclusion ${\mathfrak B} \subset {\mathfrak A}$
was proved in \cite{ZP}. Let us now prove inner regularity.
There exists a subset $X_0\subset X$, $\mu(X_0)=1$, such that the metric space
$(X_0, \rho)$ is separable.
Denote by ${\mathfrak A}_0$ the restriction of
${\mathfrak A}$ on $X_0$. Note that we may choose
$X_0$ closed in $X$ in the metric $\rho$, this implies $X\setminus X_0 \in {\mathfrak B}$.
Let $X_1$ be the completion of the metric space $(X_0,\rho)$.
Define a measure $\mu_1$ on a Borel sigma-algebra
$\tilde{\mathfrak B}_1={\mathfrak B}(X_1,\rho)$ of the Polish space
$(X_1, \rho)$ by  extending
$\mu$ from ${\mathfrak B}_0={\mathfrak B}(X_0,\rho)$ and setting $\mu_1(X_1\setminus X_0)=0$. Let ${\mathfrak B}_1$ be a completion of a sigma-algebra
$\tilde{\mathfrak B}_1$ in measure $\mu_1$.
Note that $(X_1,{\mathfrak B}_1,\mu_1)$ is a Lebesgue space as
a Polish space with Borel probabilistic measure on the completed
Borel sigma-algebra. Moreover, the map
$\mathrm{id}\colon (X_0, {\mathfrak A}_0,\mu) \to (X_1,{\mathfrak B}_1, \mu_1)$
is injective measure preserving map of Lebesgue spaces.
By lemma in p. 5 in the paper \cite{Ro} for such a map an image of
any measurable set is measurable. Thus we have
${\mathfrak A}_0 \subset {\mathfrak B}_1$. Since ${\mathfrak A}_0$
is a sigma-algebra on $X_0$, restrictions of sigma-algebras
${\mathfrak B}$ and ${\mathfrak B}_1$ on $X_0$ coincide, and $X_0 \in {\mathfrak B}$,
hence ${\mathfrak A}_0 \subset {\mathfrak B}$. Since $\mu(X\setminus X_0)=0$
we get ${\mathfrak A} = {\mathfrak B}$.
recall that any probabilistic Borel measure on a Polish space is inner
regular. Let us prove the inner regularity of the measure $\mu$ on the
(maybe not complete) metric space $(X,\rho)$.
Consider any set $A \in {\mathfrak A}$.
Then $A\cap X_0 \in {\mathfrak B}_1$, and using the inner regularity of
the measure $\mu_1$ on the Polish space $(X_1, \rho)$ we may find a compact
set $K \subset A\cap X_0$ for which
$\mu(K)=\mu_1(K)>\mu_1(A\cap X_0)-\varepsilon=\mu(A)-\varepsilon$, as desired.
\end{proof}

M. Gromov in the book \cite{G} suggests to consider arbitrary metric triples
$(X,\mu,\rho)$, which he calls $mm$-spaces. Also, Gromov asks the question about their classification,
having in mind classical situations (Riemannian manifolds and so on). It is natural to consider admissible triples
in this framework. Define equivalence of admissible triples up to measure-preserving isometries:
$(X, \mu, \rho)   \sim  (X', \mu', \rho')$,
if
$$\exists T:X\rightarrow X'; \quad T\mu=\mu'; \quad \rho'(Tx,Ty)=\rho (x,y).$$

Here is the main result on this equivalence:
\begin{theorem} {(Gromov \cite{G}; Vershik \cite{V1})}

Consider the map $F_{\rho}: X^{\infty}\times  X^{\infty}\rightarrow M_{\infty}(\Bbb R):$
$$
\quad F_{\rho}(\{x_i,y_j\}_{(i,j) \in \Bbb N\times \Bbb N})=\{\rho(x_i,y_j)\}_{(i,j) \in \Bbb N\times \Bbb N},
$$
and equip infinite product $X^{\infty}\times  X^{\infty}$ by the product-measure $\mu^{\infty}\times \mu^{\infty}$.
Let $D_{\rho}$ denote the measure on the space of matrices (i.e. random matrix of distances), which is
the $F_{\rho}$-image of the measure  $\mu^{\infty}\times \mu^{\infty}$. Call it
MATRIX DISTRIBUTION of the metric $\rho$. It is a complete invariant of above equivalence of admissible metrics.

In other words, $$(X, \mu, \rho)   \sim  (X', \mu', \rho') \Leftrightarrow   D_{\rho}= D_{\rho'}.$$
 \end{theorem}

In \cite{V4} this result is generalized to the so called \emph{pure} measurable functions of several variables.

The following  lemma is useful in the theory of admissible metrics:
\begin{Lm}\label{equivalence}
Let $\rho_1$, $\rho_2$ be admissible semimetrics on the standard space $(X, \mu)$,
and suppose that $\rho_1$ is metric. Then for any $\varepsilon>0$ there exists measurable subset $K \subset X$
such that $\mu(K)>1-\varepsilon$ and semimetric $\rho_2$ (as a function of two variables)
is continuous on $K\times K$ with respect to metric $\rho_1$.
\end{Lm}

\begin{proof}
Consider an admissible metric $\rho=\rho_1+\rho_2$.
Choose a compact set $K$ in this metric so that
$\mu(K)>1-\varepsilon$. Let's show that $\rho$ (hence $\rho_2$) is continuous
on the metric space
$(K\times K,\rho_1\times\rho_1)$.
Triangle inequality reduces this desired continuity to the following fact:
given $\delta>0$ there exists $\sigma>0$ such that $\rho(x,y)<\delta$
whenever
$x,y \in K$ and $\rho_1(x,y)<\sigma$.
If it is not true, then there exists
$\delta>0$ and two sequences $\{x_n\},\{y_n\}$ in $K$ for which
$\rho(x_n,y_n)\geq\delta$ but $\rho_1(x_n,y_n)\to 0$.
Since $(K,\rho)$ is compact, we may without loss of generality
suppose that there exist $x,y \in K$ so that
$$
\rho(x_n,x) \to 0; \qquad \rho(y_n,y) \to 0.
$$
But then
$$
\rho_1(x_n,x)\leq \rho(x_n,x)\to 0; \qquad \rho_1(y_n,y)\leq \rho(y_n,y)\to 0.
$$
Thus $\rho_1(x,y)=0$ while $\rho(x,y)\geq\delta$.
It contradicts to the assumption that $\rho_1$ is a metric.
\end{proof}

Lemma immediately implies the

\begin{Co}\label{coincidence}
Let $\rho_1$ and $\rho_2$ be two admissible metrics on the standard space $(X, \mu)$.
Then for any $\varepsilon>0$ there exists $K \subset X$ such that
$\mu(K)>1-\varepsilon$ and topologies defined by metrics $\rho_1$ and $\rho_2$ on $K$ coincide.
\end{Co}

\section{Virtual continuity}

\subsection{Luzin's theorem on measurable functions of one variable}

Furthermore we consider (measurable) real-valued functions, though most of our results remain true
for maps into standard Borel space, in particular into Polish spaces.
Egorov's and Luzin's classical theorems on measurable functions of one variable are well-known.
The generalized Luzin's theorem for arbitrary admissible triple follows from above results:

\begin{Co}[Luzin's theorem] Let  $\rho$ be an admissible metric on the standard space $(X,\mu)$,
let $f$ be a measurable map from $X$ into Polish space $(M,d)$. Then for any $\varepsilon>0$ there exists
a measurable subset $K \subset X$ such that $\mu(K)>1-\varepsilon$ and $f$ is continuous on $K$ with respect to
metric $\rho$.
\end{Co}

\begin{proof} Set $\rho_1(x,y)=\rho(x,y)+d(f(x),f(y))$. Then $\rho_1$ is a trivial example of
an admissible metric, with respect to which $f$ is continuous. By
\ref{coincidence} there exist a subset $K$ having measure $\mu(K)>1-\eps$, on which this
continuity implies continuity with respect to $\rho$.
\end{proof}

But this fact does not hold true for functions of several variables.

\subsection{Definitions and first examples}
Let $f(\cdot,\cdot)$ be a measurable function of two variables. Then Luzin's theorem analogue
(continuity on the product $X'\times Y'$ of sets of measure $>1-\eps$ with respect to given metric
$\rho[(x_1,y_1),(x_2,y_2)]=\rho_X(x_1,x_2)+\rho_Y(y_1,y_2)$) is not in general true.
This leads to the following key notion of this work.
(Sum of metrics may be replaced to maximum or other metric defining the topology of direct product.
To stress this we denote generic metric with such topology by $\rho_X \times \rho_Y$).

\begin{defe} Measurable function $f(\cdot,\cdot)$ on the product $(X,\mu)\times (Y,\nu)$ of standard spaces
is called \emph{properly virtually continuous} \textup, if for any
$\varepsilon  >0$ there exist sets $X'\subset X$\textup, $Y'\subset Y$
each of which having measure at least $1-\eps$\textup,
and admissible semimetrics $\rho_X$\textup, $\rho_Y$ on $X'$\textup, $Y'$
 respectively\textup
such that function  $f$ is continuous on  $(X'\times Y',\rho_X \times \rho_Y)$.

Function which coincides with a properly virtually continuous function
on the set of full measure in $X\times Y$ is called
\emph{virtually continuous}.
Virtually continuous functions of
 several variables are defined in the same way.
\end{defe}
It is essential that admissible metric with respect to which function becomes continuous is not arbitrary,
but respects the structure of direct product (in more general setting, it respects selected subalgebras, see further).
It is easy to verify that there does not exist universal metric of such type (i.e. such a metric that virtual
continuity implies continuity in this metric). It explains the non-trivial properties of defined notion.

It is clear that any admissible metric (considered as a function of two variables) is virtually continuous.
So is any function, which is continuous with the respect to product of admissible metrics.
Degenerated functions (or ``finite rank functions'') $f(x,y)=\sum_{i=1}^n \varphi_i(x)\psi_i(y)$, where $\phi_i(\cdot), \psi_i(\cdot)$, $i=1,\dots n$
are arbitrary measurable functions, are also virtually continuous. For the proof just use Luzin's theorem for all functions
$\varphi_i(\cdot)$,  $i=1 \dots n$, and $\psi_i(\cdot)$,  $i=1 \dots n$.

Less trivial examples of virtually continuous functions are given by functions from some Sobolev spaces
and kernels of trace class operators.
For virtually continuous functions there exist well-define restrictions on some subsets of zero measure --- concretely,
onto supporters of (quasi)bistochastic measures, see next paragraph.

An easy example of not virtually continuous measurable function on $[0,1]^2$ is provided by the
characteristic function of the triangle $\{x\geq y\}$. In general, for functions on the square of a compact group
depending of the ratio of variables the criterion of virtual continuity is simple:

\begin{proposition}\label{gruppa} Let $G$ be a metrizable compact group, $f$ be a Haar measurable function on
$G$. Then the function $F(x,y):=f(xy^{-1})$ on $G\times G$ is virtually continuous if and only if
$f$ is equivalent to a continuous function.
\end{proposition}

Stress once more that the definition of virtual
continuity is not topological, but measure-theoretical in nature.
It applies to the choice of various metrics on the measure space.
So, the direct sense of the proposition \ref{gruppa} is that the
group structure and the measure-theoretical structure allow to
reconstruct topology.

\subsection{Further properties
of virtually continuous functions}

First of all, virtually continuous functions automatically satisfy stronger
properties that are required by the definition.

At first, using Corollary \ref{coincidence} we immediately
see that metrics may be fixed a priori:

\begin{theorem}\label{arbitrarymetrics}
Let the function $f(\cdot,\cdot)$ be properly virtually continuous.
Then for any admissible semimetrics $\rho_X,\rho_Y$ on $X,Y$ and for any $\varepsilon  >0$ there exist sets $X'\subset X,  Y'\subset Y$,
each of which having measure at least $1-\eps$, such that the
function  $f$ is continuous on  $(X'\times Y',\rho_X \times \rho_Y)$.
 \end{theorem}

On the other hand choosing metrics in a special way we may force sets
$X',Y'$ from the definition to have the full measure:

\begin{theorem}\label{polnayamera}
Let function $f(\cdot,\cdot)$ be virtually continuous.
Then there exist sets $X'\subset X$, $Y'\subset Y$ of full measure and admissible semimetrics
$\rho_X,\rho_Y$ on $X',Y'$ respectively such that
 $f$ is continuous on $(X'\times Y',\rho_X\times \rho_Y)$.
 \end{theorem}

\begin{proof}
Fix admissible metrics $\sigma_X$, $\sigma_Y$ on $X,Y$ respectively.
For any $n$ use a theorem \ref{arbitrarymetrics} and find the
sets $X_n\subset X$ and $Y_n\subset Y$ of measure
at least $1-2^{-n}$ so that  $f$ is continuous
on $(X_n\times Y_n,\sigma_X\times \sigma_Y)$. Define the cut
semimetrics on $X$:
\begin{equation*}
\rho_{X;n}(x,x')=
\begin{cases} 0 & \text{if $x,x'\in X_n$ or $x,x'\notin X_n$,}
\\
1 &\text{if $x\in X_n,x'\notin X_n$ or $x'\in X_n,x\notin X_n$}.
\end{cases}
\end{equation*}
Define a set of full measure
$X'=\cup_{n=1}^{\infty} \cap_{k>n} X_k$.
Next, define the metric $\rho_X=\sigma_X+\sum_n 2^{-n} \rho_{X;n}$
(it is easy to verify that $\rho_X$ is admissible metric).
Analogously define the set $Y'$ and the metric $\rho_Y$.
Let's prove that the function $f(x,y)$ is continuous
on $(X'\times Y',\rho_X\times \rho_Y)$.
Consider converging sequences
$x_n\rightarrow x_0$, $y_n\rightarrow y_0$ in $(X',\rho_X)$, $(Y',\rho_Y)$
respectively.
Let $N$ be so large that $x_0\in X_N$, $y_0\in Y_N$.
Then convergence with respect to the semimetric $\rho_{X;n}$
implies that $x_n\in X_N$ for all large enough $n$,
analogously $y_n\in Y_N$ for large $n$. Now convergence
$f(x_n,y_n)\rightarrow f(x_0,y_0)$ follows from the continuity of
$f$ on $(X_N\times Y_N,\sigma_X\times \sigma_Y)$.
\end{proof}

Establish the following corollary:

\begin{proposition}\label{proper-coinciding} If properly
virtually continuous functions
$f(x,y)$, $g(x,y)$ coincide on the set of full
measure in $X\times Y$ , then there exist
sets $X'\subset X$, $Y'\subset Y$ of full measure
such that $f(x,y)=g(x,y)$ for all points $x\in X'$, $y\in Y'$.
\end{proposition}

\begin{proof} Apply Theorem \ref{polnayamera}
to both functions $f$, $g$. We may suppose that corresponding sets of full measure
$X'\subset X$, $Y'\subset Y$ may coincide for $f$ and $g$
(intersect sets for $f$ and for $g$), as well
as semimetrics $\rho_X$, $\rho_Y$ (sum up the semimetrics for
$f$ and for $g$). Moreover, we may suppose that $X'$, $Y'$ are supporters
of the measures
$\mu|_{X'}$, $\nu|_{Y'}$. Now note that the set of points
 $(x,y)\in X'\times Y'$, for which
$f(x,y) \ne g(x,y)$, is open in $X'\times Y'$, hence it
either have a positive measure (but this impossible by our
assumption) or is empty (as desired).
\end{proof}

Thus all properly virtually continuous functions, which are equivalent to a
given virtually continuous function, coincide on a product of sets  having full measure.

It is useful to think about a function of two variables on $X\times Y$
as a map from  $X$ to the space of functions on $Y$
(i.e. $f(x,y)\equiv f_x(y)$). See details on using this viewpoint for classification
of measurable functions in \cite{VS}.
Virtual continuity may be expressed in these terms by the following equivalent definition:

\begin{theorem}\label{compactness} The following properties
of a function $f(\cdot,\cdot)$ are equivalent\textup:
\begin{itemize}
\item[\textup{(i)}] $f$ is virtually continuous\textup;
\item[\textup{(ii)}] for any $\varepsilon>0$ there exist sets
$X'\subset X,  Y'\subset Y$ having measure not less than $1-\eps$ and
a semimetric $\rho_Y$ on $Y'$\textup,
so that the function $f_x(\cdot)$ is equivalent to a continuous
function on $(Y',\rho_Y)$  for almost all $x\in X'$\textup;
\item[\textup{(iii)}] for any$\eps>0$ there exist sets
$X'\subset X$\textup,
 $Y'\subset Y$ having measure not less then $1-\eps$
such that the set of functions $f_x(\cdot)$ on $Y'$
\textup(where variable $x$  runs over $X'$\textup)
is totally bounded
\textup(precompact\textup) as a metric subspace in $L_{\infty}(Y')$;
\item[\textup{(iv)}] for any $\eps>0$ there exist
sets $X'\subset X$\textup,
 $Y'\subset Y$ having measure not less then $1-\eps$ such that
the set of functions $f_x(\cdot)$ on $Y'$
\textup(where variable $x$  runs over $X'$\textup)
is separable
\textup(precompact\textup) as a metric subspace in $L_{\infty}(Y')$.
\end{itemize}
\end{theorem}

\begin{proof} Clearly (i) implies (ii) and (iii) implies (iv).

Let's prove (iii) assuming (ii).
Removing appropriate sets of zero measure from $X'$, $Y'$ we
may suppose that the function $f(x, \cdot)$ is equivalent
to a continuous function on $Y'$ for any $x\in X'$ and the
function $f(y, \cdot)$ is equivalent
to a continuous function on $X'$ for any $y\in Y'$.
Choose a compact subset  $Y_1\subset Y'$
such that $\nu(Y_1)>1-2\varepsilon$.
Next, replace  $Y_1$ to a supporter
$\supp(\nu|_{Y_1})$ of the measure $\nu$ restricted to $Y_1$.
Now for continuous functions on $Y_1$ teh distances in
$C(Y_1)$ and in $L_{\infty}(Y_1)$ coincide.
Let $\cal{S}$ be a countable family of open balls
in $Y_1$, which form a base of topology.
For $x\in X'$ denote by $f'_x(\cdot)$ a continuous
function on $Y_1$, which is equivalent to $f(x,\cdot)$.
Let's prove that a map $\Phi\colon x\mapsto f'_x(\cdot)$
is measurable as a map from  $X$ into $C(Y_1)$
(with Borel sigma-algebra).
It suffices to check that a preimage of a ball
$$
X_g:=\{x\in X\colon  \forall y\in Y_1 \quad |f'_x(y)-g(y)|\leq r\}
$$
is measurable for any continuous function $g\in C(Y_1)$
and any positive $r$. Note that an inequality
$|f'_x(y)-g(y)|\leq r$ holds for all
 $y\in Y_1$ if and only if all inequalities
$$
\frac1{\nu(B)}\left|\int_B f'_x(y)d\nu(y)-\int_B g(y) d\nu(y)\right|\leq r
$$
for balls $B\in {\cal S}$ hold.
But under integral sign the function
$f'_x(\cdot)$ may be replaced to an equivalent function $f(x,\cdot)$,
and a map  $x\mapsto \int_B f(x,y)d\nu(y)$ is measurable in $x$.
Thus the set $X_g$ is measurable as a countable intersection of
measurable sets.

So $\Phi$-preimage of the measure $\mu$ on $X'$ is a Borel measure on $C(Y_1)$.
The space $C(Y_1)$ is complete separable, hence this measure is inner
regular and there exists a compact set $K\subset C(Y_1)$
such that $\mu(\Phi^{-1}(K))>1-2\eps$. But $\eps$ is arbitrary,
thus the function $f(x,y)$ satisfies (iii).

It remains to show virtual continuity of  $f$ assuming (iv).
Fix $\eps>0$. Choose $X',Y'$ as in (iv).
Denote by $K$ a separable in $L_{\infty}(Y')$ family of functions of the form
$f(x,\cdot)$. Let $K'$ be a countable dense subfamily of $K$.
Let $\rho$ be an admissible semimetric on $Y'$ such that any function from $K'$
is $\rho$-continuous
(such a metric may be constructed, for example, as a uniformly convergent series
for metrics taken for specific functions from $K'$).
Let $Y_1\subset Y'$ be a compact set in $\rho$ having measure at least $1-2\eps$.
Replace $Y_1$ (if necessary) to the supporter of
the restriction of the measure $\nu$ on $Y_1$. Now any open ball in $Y_1$
has a positive measure. It implies that for continuous functions on $Y_1$
distances in $C(Y_1)$ and in $L_{\infty}(Y_1)$ coincide.
Thus any function $f(x,\cdot)$, $x \in X'$, is equivalent to a unique continuous
function $f'(x,\cdot)$ on $Y_1$.
Define a semimetric on  $X'$ by the formula
$$\rho_{X'}(x,x')=\|f'(x,\cdot)-f'(x',\cdot)\|_{C(Y_1)}.$$
It is admissible since $C(Y_1)$ is separable. The function $f'$ is continuous on
$(X'\times Y_1,\rho_{X'}\times \rho)$. Indeed, if $x_n\rightarrow x_0$,
$y_n\rightarrow y_0$, then
$$
|f'(x_n,y_n)-f'(x_0,y_0)|\leq |f'(x_n,y_n)-f'(x_0,y_n)|+|f'(x_0,y_n)-f'(x_0,y_0)|,
$$
and both summands tend to 0. The function $f'$
is properly virtually continuous on $X' \times Y_1$.
By Fubini theorem it is equivalent to the initial function $f$ on
$X'\times Y_1$. Since $\eps>0$ it implies that  $f$ is virtually continuous.
\end{proof}

It's remarkable that the spaces $X$ and $Y$ (i.e. arguments of the function) play different roles in (ii--iv).
However, a posteriori the property appears to be symmetric under the change of order of variables.
This is another demonstration of the non-triviality of the virtual continuity concept.

Above characteristics of virtual continuity easily imply that virtual continuous functions
form a nowhere dense subset in the space of all functions of two variables (with measure convergence topology).

\subsection{Virtual topology}

A function is measurable if and only if for any open set its preimage is open.
Virtual continuity of a function of two variables admits a similar definition.

\begin{defe} A measurable set $Z\subset X\times Y$ is called virtually open, if
for some subsets
 $X'\subset X$, $Y'\subset Y$ of full measure the set
$Z \cap (X'\times Y')$ is a countable union of measurable rectangles $R_i=A_i\times B_i$,
$i=1,2,\dots$. Set is called virtually closed if it complement is virtually
open.
\end{defe}

The sense (and the name) of this concept is explained by the following

\begin{lemma}\label{virtual-structure}
1) Let $X'$, $Y'$ be sets of full measure in $X$, $Y$ respectively;
$\rho_X$, $\rho_Y$ be admissible semimetrics on $X'$, $Y'$; a set $Z\subset X\times Y$
is so that $Z\cap X'\times Y'$ is open in $(X'\times Y',\rho_X\times \rho_Y)$. Then the set
$Z$ is virtually open.

2) Vice versa, for any virtually open set in $X\times Y$ there exist appropriate
$X'$, $Y'$ of full measure and admissible semimetrics $\rho_X$, $\rho_Y$ such that
$Z\cap X'\times Y'$ is open in $(X'\times Y',\rho_X\times \rho_Y)$.
\end{lemma}

\begin{proof} 1). Replacing set $X'$, $Y'$ onto appropriate subsets of full measure we may
suppose that $(X',\rho_X)$, $(Y',\rho_Y)$ are separable semimetric spaces, in which all balls
are $\mu$-, $\nu$-measurable respectively. Then topologies of those
spaces have separable bases consisting measurable sets, thus the topology
of direct product has a countable bases consisting measurable rectangles.
This just implies that an open subset of $X'\times Y'$ is a countable
union of measurable rectangles, hence it is virtually open.

2) For a given countable family of measurable
sets in  $X$ we may construct an admissible semimetric, in which they
are all open (a possible construction is to sum up cut-metrics with rapidly decaying
coefficients). Doing this for $X$-projections of countably many rectangles whose union
is our virtually open set we construct a semimetric on $X$ (strictly speaking, on $X'$), analogously on $Y$.
\end{proof}

\begin{theorem}\label{virtual} A measurable function $f(x,y)$ on $X\times Y$
is properly virtually continuous if and only if  $f$-preimage of any open
set on the real line is virtually open.
\end{theorem}

\begin{proof}
(``Only if'' part.) If a function $f$ is properly virtually open, then by the theorem \ref{polnayamera}
it is continuous on the set product $(X'\times Y',\rho_X,\rho_Y)$, where $X'\subset X$, $Y'\subset Y$ have full measure
and $\rho_X,\rho_Y$ are admissible semimetrics. Then $f$-preimages of open sets on the real line have a structure
described in \ref{virtual-structure}.

(``If'' part.) Assume that $f$-preimage of any rational interval on the real
line is virtually open. Then there exist sets of full measure $X',Y'$
such that for the restriction $f:X'\times Y'\rightarrow \mathbb{R}$
all those preimages are countable unions of measurable rectangles in $X'\times Y'$
($X'$, $Y'$ may be defined as intersections of corresponding sets of full measure).
Considering series of cut semimetrics we may easily construct admissible
metrics on $X'$, $Y'$ such that each of those countably many rectangles is open.
Then the function $f$ is continuous with respect to a pair of constructed
admissible metrics on $X'\times Y'$, hence it is virtually continuous.
\end{proof}

 \subsection{Thickness}
Consider the space $X\times Y$ with product measure $\mu \times \nu$.
Choose two subalgebras in its sigma-algebra, defined by projections onto $X$ and $Y$.
We write $A\subsetmod B$ if $A,B \subset X\times Y$ and $\mu\times\nu(B\setminus A)=0$.
Also, we write $\leqmod$ or $\geqmod$, if the corresponding inequality holds
$\mu\times \nu-$almost everywhere.

\begin{defe}
For a measurable set $Z\subset X\times Y$ define its \emph{proper thickness} as
\begin{equation}\label{thickness}
\sthi(Z)=\inf\{\mu(\tilde X)+\nu(\tilde Y) \colon \tilde X \subset X,\, \tilde Y\subset Y,\, Z\subset (\tilde X\times Y)\cup (X\times \tilde Y)\}.
\end{equation}
The \emph{thickness} of a set $Z$ is defined as
$$
\thi(Z)=\inf\{\mu(\tilde X)+\nu(\tilde Y) \colon \tilde X \subset X,\, \tilde Y\subset Y,\, Z\subsetmod (\tilde X\times Y)\cup (X\times \tilde Y)\}.
$$
In other words,
\begin{equation}\label{not-proper-thickness}
\thi(Z)=\min\{\sthi(Z')\colon Z\subsetmod Z'\}.
\end{equation}
\end{defe}

Minimum in \eqref{not-proper-thickness} is always attained, because we may intersect a minimizing
sequence of sets.

The sets $\tilde X\times Y$, $X\times \tilde Y$ are just sets from chosen
subalgebras, so we may extend our definition to other choices of selected
subalgebras in a standard space.

The following properties of thickness are immediate:
\begin{itemize}
\item thickness of a set does not exceed 1 and equals 0 for and only for sets of measure 0;
\item thickness of a subset does not exceed a thickness of a set;
\item thickness of a set is not less than its measure;
\item thickness of a finite or countable union of sets does not exceed sum of thicknesses.
\end{itemize}

The following lemma is not hard too:

\begin{lemma}\label{virtual-thickness}  If a set $Z\subset X\times Y$ is virtually open, then
$\thi(Z)=\sthi(Z)$.
\end{lemma}

\begin{proof} It suffices to prove that $\sthi(Z)\leq \mu(\tilde X)+\nu(\tilde Y)$,
if $Z\subsetmod (\tilde X\times Y)\cup (X\times \tilde Y)$.
Replacing $X$, $Y$ to subsets of full measure we may suppose that $Z=\cup_{i=1}^{\infty} A_i\times B_i$.
Note that if
\begin{equation}\label{rectanglemod0}
A_i\times B_i\subsetmod (\tilde X\times Y)\cup (X\times \tilde Y),
\end{equation}
then either $A_i\subsetmod \tilde{X}$ or $B_i\subsetmod \tilde{Y}$.
In both cases we may add sets of 0 measure to
$\tilde{X}$, $\tilde{Y}$ so that $\subsetmod$ in \eqref{rectanglemod0}
becomes just $\subset$. Doing this for all
$i=1,2,\dots$ successively we get the desired inequality.
\end{proof}

The following lemma provides an equivalent and sometimes more useful
definition of the thickness.

\begin{lemma}\label{thick2}
For any $Z\subset X\times Y$ consider pairs of measurable functions
$f\colon X \to [0,1],$ $g\colon Y\to [0,1],$ for which
$f(x)+g(y)\geq \chi_{Z}(x,y)$ \textup(resp. $f(x)+g(y)\geqmod \chi_{_Z}(x,y)$\textup).
Then the proper thickness \textup(resp. thickness\textup) of the set
$Z$ is the infimum of $\int_X f d\mu+\int_Y g d\nu$. Moreover\textup,
this infimum is realized as well as infimum in \eqref{thickness}.
\end{lemma}
\begin{proof}
Clearly, if sets $\tilde X \subset X$ and $\tilde Y \subset Y$ are such that
$Z\subset (\tilde X\times Y)\cup (X\times \tilde Y)$, then $f=\chi_{_{\tilde X}}$ and $g= \chi_{_{\tilde Y}}$
satisfy inequality
$f(x)+g(y)\geq \chi_{_Z}(x,y)$.
Thus we just need to prove that $\int_X f +\int_Y g \geq \sthi(Z)$ whenever
$f(x)+g(y)\geq \chi_{_Z}(x,y)$.
For any $t\in [0,1]$ define the sets
$X_t=\{x \in X\colon f(x)\geq t\}$ and $Y_t=\{y \in Y\colon g(y)\geq t\}$. Clearly, if $f(x)+g(y)\geq 1$, then
for any $t$ either
$x\in X_t$ or $y \in Y_{1-t}$. Thus for any $t$ we have
$
\chi_{_Z}(x,y) \leq \chi_{_{X_t}}(x)+\chi_{_{Y_{1-t}}}(y)
$.
Therefore $\sthi(Z) \leq \mu(X_t)+\nu(Y_{1-t})$. Integrating by $t$ we get
$$
\sthi(Z)\leq \int\limits_0^1  \mu(X_t)+\nu(Y_{1-t}) dt = \int\limits_X f + \int\limits_Y g.
$$
Moreover, if $\sthi(Z)=\int\limits_X f + \int\limits_Y g$, then for almost all $t\in [0,1]$
the infimum in \eqref{thickness} is realized on the pair of sets $(X_t,Y_{1-t})$.

It suffices to prove that there exists a minimizing pair of functions. Consider a minimizing
sequence $(f_n(x),g_n(y))$, $\int f_n+\int g_n\rightarrow  \sthi(Z)$.
By known Koml\'os theorem \cite{JK} we may suppose that a sequence $f_n':=(f_1+\dots+f_n)/n$
converges to some function $f$ almost everywhere in $X$, and  $g_n':=(g_1+\dots+g_n)/n$ converges to $g$ almost everywhere in $Y$\footnote{Only simple
version of Koml\'os theorem, in which functions are uniformly bounded, is used.}.

So, we may suppose that $f(x)=\limsup_n f_n'(x)$ for all $x\in X$, and
$g(y)=\limsup_n g_n'(y)$ for all $y\in Y$. It follows that
 $f(x)+g(y)\geq \chi_Z(x,y)$ for all  $x\in X$, $y\in Y$, hence a pair
$(f,g)$ is minimizing.
\end{proof}

Using lemma \ref{thick2} we may establish ``continuity of thickness from below'':
\begin{lemma}
Let $\{Z_n\}$ be an increasing sequence of measurable sets\textup, $Z=\cup_n Z_n$. Then $\thi(Z)=\lim\thi(Z_n),$
$\sthi(Z)=\lim \sthi(Z_n)$.
\end{lemma}
\begin{proof}

Clearly $\thi(Z)\geq \thi(Z_n)$ for all $n$, hence $\thi(Z)\geq \lim\thi(Z_n)$.
Let's prove an opposite inequality. We start with functions $f_n\colon X \to [0,1]$,  $g_n\colon Y \to [0,1]$ so
that
$f_n(x)+g_n(y)\geqmod \chi_{_{Z_n}}(x,y)$, and
$\int_Xf_n+\int_Yg_n \leq \thi(Z_n)+1/n$.
Any bounded sequence in $L^2$ contains a weakly convergent subsequence, using this twice we may suppose that the sequence
$f_n$ weakly converges to $f$ in $L^2(X, \mu)$, and $g_n$ weakly converges to $g$ in $L^2(Y, \nu)$.
Then $f_n(x)+g_n(y)$ converges to $f(x)+g(y)$ weakly in $L^2(X\times Y, \mu\times \nu)$.
Since weak limit preserves inequalities we have $f\colon X \to [0,1]$ and $g \colon Y \to [0,1]$.
Moreover, for any $n$ we have
$$
f(x)+g(y)\geqmod \chi_{_{Z_n}}(x,y).
$$
Hence 
$$
f(x)+g(y)\geqmod \chi_{_Z}(x,y).
$$
But $\int_Xf+\int_Yg=\lim(\int_Xf_n+\int_Yg_n)\leq \lim \thi(Z_n)$, thus $\thi(Z)\leq \thi(Z_n)$.

For the proper thickness we would replace in the above prove weak convergence in $L^2$
to the almost everywhere convergence obtained from the Koml\'os theorem \cite{JK}.
\end{proof}

Note that upper continuous of thickness does not hold:
all sets $\{(x,y)\colon 0<|x-y|<1/n\}\subset [0,1]^2$ have thickness ~1, but their intersection is empty.

Now we define a convergence of functions ``in thickness'' analogously to convergence
``in measure''. This is a convergence in the following metrizable
topology:

\begin{defe}
Define a distance $\tau(f(\cdot,\cdot),g(\cdot,\cdot))$ between two
arbitrary measurable functions as infimum of such $\eps>0,$ for which
$$
\thi(\{(x,y)\colon |f(x,y)-g(x,y)|>\eps\})\leq \eps.
$$
 \end{defe}

Convergence in this $\tau$-metrics implies convergence in measure (but not vice versa).

\begin{lemma}
The set of measurable functions is complete in the $\tau$-metric.
\end{lemma}
\begin{proof}
Consider a sequence  $\{f_n(\cdot,\cdot)\}$ of measurable functions, which is
fundamental in the $\tau$-metric. Passing to a subsequence we may suppose that
$\|f_n-f_{n+1}\|_{\tau}<2^{-n}$.
Put $Z_n=\{(x,y)\colon |f_n(x,y)-f_{n+1}(x,y)|>2^{-n}\}$. Then $th(Z_n)\leq 2^{-n}$, and
for the set $Z_n':=\cup_{k\geq n} Z_k$ we have
$\thi(Z_n')\leq 2^{1-n}$. Thus the set $\cap Z_n'$ has zero thickness,
while outside this set the sequence $(f_n)$ converges pointwise to some function $f_0$.
Moreover, outside $Z'_n$ this sequence converges uniformly and
$$
|f_0-f_n|\leq |f_n-f_{n+1}|+|f_{n+1}-f_{n+2}|+\dots\leq 2^{1-n}.
$$
 It means that $\|f_0-f_n\|_{\tau}\leq 2^{1-n}$, hence $f_n$ converges to $f_0$
in the metric $\tau$.
\end{proof}

Let $\xi_X: X=\sqcup_{i=1}^n X_i$, $\xi_Y: Y=\sqcup_{i=1}^m Y_i$ be finite partitions of the spaces $X,Y$ respectively
onto measurable subsets of positive measure.
Functions which are constant mod $0$ on each product $X_i\times Y_j$ are called \emph{step functions}.
Finite linear combinations $\sum_{i=1}^N a_i(x)b_i(y)$ are called \emph{functions of finite rank}.

The following theorem connects finite rank functions and virtual continuity.

\begin{theorem}\label{tau-virtual}
The $\tau$-closure of the set of step functions \textup(or the set of finite rank functions\textup)
is exactly the set of virtually continuous functions. In other words\textup,
a function $f$ is virtually continuous on $X\times Y$ if and only if for any $\eps>0$ there exist
such families of disjoint measurable sets $A_1,\dots,A_n\subset X,$
$B_1,\dots,B_n \subset Y,$ and numbers $c_{ij},1\leq i,j \leq n,$
that $\sum \mu(A_i)>1-\eps,$ $\sum \nu(B_i)>1-\eps,$ $|F(x,y)-c_{ij}|<\eps$ for almost all $x\in A_i,$
$y\in B_j$.
\end{theorem}

\begin{proof}
At first, we show that a $\tau$-limit $f(x,y)$ of step functions  $f_n(x,y)$ is virtually
continuous.
We choose admissible metrics $\rho_X$ and $\rho_Y$ on $X$, $Y$ respectively so that step functions $f_n$
are continuous.

Passing to a subsequence we may suppose that $\tau(f_n,f)<\frac{1}{2^n}$. It means that
$$
\thi(\{(x,y)\colon |f(x,y)-f_n(x,y)|>\frac{1}{2^n}\})<\frac{1}{2^n}.
$$
Choose subsets $X_n\subset X$,  $Y_n \subset Y$ so that
$$
\Big\{(x,y)\colon |f(x,y)-f_n(x,y)|>\frac{1}{2^n} \Big\} \subsetmod (X\times Y_n) \cup (X_n\times Y),
$$
and  $\mu(X_n)+\nu(Y_n)<\frac{1}{2^n}$. Define sets $\tilde X_n=X\setminus \cup_{k>n}X_k$
and $\tilde Y_n=Y\setminus \cup_{k>n}Y_k$. Clearly $\mu(\tilde X_n)>1-\frac{1}{2^n}$ and $\nu(\tilde Y_n)>1-\frac{1}{2^n}$.
We may replace the set $\tilde X_n$ to its subset (call it $\tilde X_n$ again) which is compact in metric $\rho_X$,
coincides with the supporter of the measure $\mu$ restricted to $\tilde X_n$ and has large measure
$\mu(\tilde X_n)>1-\frac{1}{2^n}$. Do the same with $\tilde Y_n$.
Now $L_\infty$ and $C$ define the same distance between continuous functions
on $\tilde X_n \times \tilde Y_n$, hence $(f_k)$ is a sequence of functions
converging in $C(\tilde X_n\times \tilde Y_n)$.
Hence the function $f$ is equivalent to a continuous function on $\tilde X_n \times \tilde Y_n$.
Taking into account that $n$ is arbitrary we conclude that $f$ is virtually continuous.

Now we prove the converse. Let $f$ be virtually continuous. We need to approximate $f$ by step functions
in $\tau$-metric. Choose any $\eps>0$ and find sets $\tilde X \subset X$, $\tilde Y \subset Y$ and admissible metrics
$\rho_X$, $\rho_Y$ such that $f$ is equivalent to a function $\tilde f$, which is continuous on $\tilde X \times \tilde Y$, and $\mu(\tilde X)>1-\eps$,
$\nu(\tilde Y)>1-\eps$. Passing to subsets we may also suppose that $\tilde X$ and $\tilde Y$ are compact in respective
metrics. using uniform continuity of $\tilde f$ on a compact metric space we partition $\tilde X$ and $\tilde Y$
onto small enough parts so that $\tilde f$ is constant up to $\eps$ on products
of partition elements. This provides a step function which is $\eps$-close to $f$ in $\tau$-metric.
\end{proof}

Theorem \ref{tau-virtual} also shows a purely measure-theoretical character of virtual continuity
and possibility to generalize it for other pairs of sigma-subalgebras. Close things are discussed in \cite{S}.

We apply theorem \ref{tau-virtual} for proving Proposition \ref{gruppa}.

\begin{proof}[Proof of Proposition \ref{gruppa}]
If a function $f$ is equivalent to a continuous function, then the function $F$ on $G\times G$
is equivalent to a continuous (in metric of $G$, which is admissible w.r.t. Haar measure $|\cdot|$). Hence $F$ is
virtually continuous.

Let's prove the opposite. Fix $\eps>0$. Choose any point $g_0\in G$
and establish that in its small enough neighborhood essential variance of $f$
does not exceed $\eps$. Since $g_0$ and $\eps>0$ are arbitrary we see that $f$
coincides almost everywhere with its essential upper limit, which itself
is continuous function.

Using virtual continuity of $f$ we find such families of disjoint measurable subsets $A_1,\dots,A_n$,
$B_1,\dots,B_n$, and numbers $c_{ij},1\leq i,j \leq n$,
that $\sum |A_i|>1/2$, $\sum |B_i|>1/2$, $|F(x,y)-c_{ij}|<\eps/3$ for almost all $x\in A_i$,
$y\in B_j$. By pigeonhole principle we find indices  $i,j$ such that $|C|>0$, where $C=A_i\cap g_0B_j$.

Continuity of shift in mean implies that for some neighborhood $\Delta$ of unity we have
$|C\cap sC|>|C|/2$ for all $s\in \Delta$. It suffices to prove that in the neighborhood $\Delta g_0$
of a point $g_0$ inequality $|f(z)-c_{ij}|<\eps/3$ holds for almost all points $z$.
Let $Sg_0$, where $S\subset \Delta$, be a set of $z$, for which it is not so.
Then for almost all $x\in C\subset A_i$, $y\in g_0^{-1}C\subset B_j$
we have $xy^{-1}\notin Sg_0$. In other words, the following integral vanishes:
\begin{align*}
0=\int \chi_{_C}(x) \chi_{_{g_0^{-1}C}} (y) \chi_{_{Sg_0}}(xy^{-1}) dx dy=\int \chi_{_C}(x) \chi_{_{tC}}(x) \chi_{_S}(t) dx dt=
\\
=\int_S |C\cap tC| dt\geq |S|\cdot |C|/2,
\end{align*}
hence $|S|=0$, as desired (second equality corresponds to the change
of variables$(x,y)\mapsto (x,t),t=xy^{-1}g_0^{-1}$).
\end{proof}

Measurable functions $f(\cdot,\cdot)$, as we have seen, are classified by matrix distributions,
i. e. by measures on the space of infinite matrices $(a_{ij})_{i,j=1}^\infty$, induced by the map
$f\mapsto (a_{ij}=f(x_i,y_j))$, where points $x_i$ in $X$ and $y_i$ in $Y$, $i=1,2,\dots$,
are chosen independently. Virtual continuity also may be characterized on this manner:

\begin{theorem}
Let $x_1, x_2, \dots$ (resp. $y_1, y_2, \dots$) be independent random points in $X$ (resp. in $Y$).
Virtual continuity of the measurable function $f(x,y)$ is equivalent to each of two following conditions:
\begin{itemize}
\item[\textup{(i)}]
For any $\varepsilon >0$
there exists a positive integer $N$ such that the probability of the following event tends 1
when $n$ grows:
if points $x_1,\dots,x_n$ are chosen independently at random in $X$\textup, $y_1,\dots,y_n$ are chosen independently at random
 in $Y$\textup, then there exist partitions
$\{1,\dots,n\}=\sqcup_{i=0}^N A_i=\sqcup_{i=0}^N B_i$\textup such
 that $$|A_0|<\varepsilon  n, \quad |B_0|<\varepsilon n, \quad |f(x_s,y_t)-f(x_r,y_p)|<\varepsilon $$
 for $n\geq i,j>0$\textup, $s,r\in A_i, p,t\in B_j$.

\item[\textup{(ii)}] For any $\varepsilon >0$ there exists a positive integer
 $N$\textup, for which the probability of the following event equals 1\textup:

 if points $x_1, x_2, \dots$ are chosen
independently at random in $X$\textup, $y_1,y_2,\dots$ are chosen independently at random
 in $Y$\textup, then there exist two partitions of the naturals
 $\{1, 2, \dots,\}=\sqcup_{i=0}^N A_i=\sqcup_{i=0}^N B_i$\textup,
 so that upper density of the set $A_0\cup B_0$ is less than $\eps$ \textup(i.e. $\limsup |(A_0\cup B_0)\cap [1,n]|/n<\eps$\textup) and
 $|f(x_s,y_t)-f(x_r,y_p)|<\varepsilon $
 for $i,j>0$\textup, $s,r\in A_i, p,t\in B_j$.
\end{itemize}
\end{theorem}

\begin{proof}
Virtually continuous function may be approximated in $\tau$-metric by step functions,
hence satisfies (i), (ii) by Law of Large Numbers.

Deduce (i) from (ii).
A set of upper density less than $\eps$ contains less than $2\eps n$ elements from 1 to $n$
for all large enough $n$. It means that with probability 1 for all large enough $n$ there exist
partitions of the set $\{1,\dots,n\}$, which satisfy (i). It certainly implies
that a probability of a specific event tends to 1 as a function of $n$.

It remains to deduce virtual continuity from
(i). We may and do suppose that $X$, $Y$ are both unit segments $[0,1]$ with Lebesgue
measure.

We need the following standard lemma on large deviations for
$U$-statistic:

\begin{lemma}
Let $Z\subset X\times Y$ be a measurable subset of $X\times Y$.
Consider the following event: ``number of points $(x_i,y_j)$\textup, $1\leq i,j\leq n$ in $Z$
differs from $n^2|Z|$ on more than $n^{9/5}$.'' It's probability does not exceed $2n^{-3/5}$.
\end{lemma}

\begin{proof}
Define a function $g(x,y)=\chi_{_Z}(x,y)-|Z|$ (with zero average)
and random variables $\xi_{i,j}=g(x_i,y_j)$. We have $n^2$ centered random variables taking values in $[-1,1]$.
Most of them are independent, this allows to estimate the variance
\begin{multline*}
\mathbb{E}\left(\sum g(x_i,y_j)\right)^2=\mathbb{E}\Big(\sum_{i,j} g(x_i,y_j)^2 +\sum_{i,j,k\ne j} g(x_i,y_j)g(x_i,y_k)+\\
+\sum_{i,j,k\ne i}g(x_i,y_j)g(x_k,y_j)
\Big)\leq 2n^3.
\end{multline*}
Now required estimate follows from the Chebyshev inequality.
\end{proof}

Choose a subset ${\cal N}\subset \mathbb{N}$ so that $\sum\limits_{n\in {\cal N}}n^{-3/5}<+\infty$.
Then by Borel--Cantelli lemma for any measurable $Z\subset [0,1]^2$
for almost all pairs of sequences  $(\{x_i\},\{y_i\})$ for large enough  $n \in {\cal N}$ we have
\begin{equation}\label{inZ}
\frac{n^2}{2}|Z|\leq \#\{(i,j)\colon 1\leq i,j\leq n, (x_i,y_j)\in Z\}\leq 2n^2|Z|.
\end{equation}

Apply this for a countable family of sets of type $Z=R\setminus f^{-1}(\Delta)$, where $R$ runs over rectangles with rational
coordinates of vertices, and $\Delta$ runs over rational intervals on a line. We get
for almost any random pair of sequences $(\{x_i\},\{y_i\})$ inequality
\eqref{inZ} for any $Z$ for large enough (how large depends on $Z$) $n$. In further we consider only pairs
of sequences satisfying this property.

Fix $\eps>0$ and find $N$ from (i). Consider a random pair of sequences $x_1,x_2,\dots\in X$ and $y_1,y_2,\dots \in Y$.
With probability 1 this sequence satisfies (i) for each $n$.
For fixed $n \in {\cal N}$ consider empiric distributions $\mu_j(n)=n^{-1}\sum_{i\in A^n_j} \delta(x_i)$ on $X$
and analogous empiric distributions on $Y$.
Passing to a subsequence we suppose that a sequence of measures $\{\mu_j(n)\}_{n \in {\cal N}}$
weakly converges when  $n \to +\infty$ to some measures $\mu_j(\infty)$. We may also suppose that corresponding measures on
$Y$ converge. Note that $\sum_{j=1}^N \mu_j(\infty)=\mu$ with probability 1, hence all measures $\mu_j(\infty)$
are absolutely continuous w.r.t $\mu$ with probability 1, they have Radon--Nikodym densities
$\varphi_j$. We have
$\int \varphi_0\leq \eps$, thus the measure of the set of $x\in X$ satisfying
$\varphi_0(x)\geq 1/2$ does not exceed $2\eps$. For any other
$x$ we have $\sum_{j=1}^N \varphi_j(x)\geq 1/2$,
hence $\varphi_j(x)\geq 1/2N$ for some $j$. Thus we may partition
$X$ onto sets $X_0,X_1,\dots,X_N$ so that
$\mu (X_0)<2\eps$, $\varphi_j\geq 1/2N$
on $X_j$. Construct analogous partition on~$Y$.
Let's prove that with probability 1 function $f$
has essential variation (essential supremum minus
essential infimum) at most
$2\eps$ on $X_j\times Y_k$ for $j,k\geq 1$. Indeed, if essential variation exceeds
$2\eps$, then there exists rational intervals
$\Delta_1,\Delta_2$ at distance at least
$3\eps/2$ such that the sets $f^{-1}(\Delta_i)$, $i=1,2$ intersected with $X_j\times Y_k$
have positive measure. Fix small $\delta>0$ (to be chosen later dependently on $\eps$ and $N$).
There exist rational rectangles $R_i$, $i=1,2$ so that
$$
|R_i\cap (X_j\times Y_k)\cap f^{-1}(\Delta_i)|>(1-\delta)|R_i|.
$$
Weak convergence of measures, inequalities $\varphi_j \geq 1/2N$ on $X_j$ and analogous inequalities on $Y_k$
imply that when $n \to \infty$, $n \in {\cal N}$, we have
\begin{multline*}
\mu_j(n)\times \nu_j(n)(R_i)\to \mu_j(\infty)\times\nu_j(\infty)(R_i)\geq \\
\geq \mu_j(\infty)\times\nu_j(\infty)(R_i\cap (X_j\times Y_k))\geq \frac{1}{4N^2}|R_i \cap (X_j\times Y_k)|,
\end{multline*}
hence for large enough $n \in {\cal N}$ the number of points $(x_s,y_t)\in R_i$ with $s \in A^n_j$, $t \in B^n_k$ is not less than
$$
\frac{1-\delta}{10 N^2}|R_i| n^2.
$$
Property (i) guarantees that either all those points
(for $i=1$ and $i=2$ together) do not lie in $f^{-1}(\Delta_1)$, or all of them do not lie in $f^{-1}(\Delta_2)$.
But for each of the sets
$Z=R_i\setminus f^{-1}(\Delta_i)$ of measure at most $\delta |R_i|$ for large enough  $n\in {\cal N}$ the number of points $(x_s,y_t)$, $1\leq s,t\leq n$,
in $Z$ does not exceed $2 \delta |R_i| n^2$. This contradicts to above lower estimate when $\delta$ is small enough.

Thus with probability 1 function $f$ may be approximated by a step function (respecting constructed partitions)
in $\tau-$metric. It suffices to use that $\eps$ is arbitrary and apply Theorem \ref{tau-virtual}.
\end{proof}

\subsection{Bistochastic measures and polymorphisms}\label{bsm}

From the measure-theoretical point of view a function of $k$ variables
on the product of standard continuous spaces is nothing but the function on
the standard continuous space (due to isomorphism of all such spaces).
In order to deal with it as a function of $k$ variables, we have to introduce another
category, then just measurable spaces.

Namely, consider the following structure: the measure space $({\cal X},{\mathfrak A}, m)$,
with $k$ selected sigma-subalgebras ${\mathfrak A_1},\dots,{\mathfrak A_k}$ in $\mathfrak A$.
It is natural to suppose that those subalgebras generate the whole sigma-algebra
$\mathfrak A$.

The connection with general viewpoint is the following: in the space
${\cal X}=\prod_{i=1}^k (X_i,{\frak A_i},\mu_i)$, $m=\prod \mu_i$,
identify algebras ${\frak A_i}$ with subalgebras of
${\frak A}=\prod {\frak A}_i$ by multiplying to trivial
subalgebras on other multiples. In other words, function $f(x_1,\dots,x_k)$ on ${\cal X}$
is ${\frak A}_i$-measurable iff $f$ depends only on $i$-th variable $x_i$ ($i=1, \dots, k$).
Functions depending on any less numbers of variables are defined similarly.

 \begin{defe}
 A measurable function on ${\cal X}$ with $k$ selected
subalgebras is called a \it{general measurable function
of $k$ variables.}
\end{defe}

In the classical case those subalgebras are independent and variables are called
independent their-selves  \footnote{though this is
just a lucky coincidence of probabilistic and analytical meanings of independence
}, but many fact on measurable functions remain true in general case aswell.

  Consider a measure $\lambda$ on sigma-algebra $\mathfrak A$.
It may be restricted onto sigma-subalgebras ${\mathfrak A}_i$, $i=1,\dots,k$.
Consider such measures $\lambda$ that those restrictions
are absolutely continuous with respect to restrictions of $m$ onto ${\mathfrak A}_i$.
 If restrictions of  $\lambda$ onto ${\frak A}_i$ coincide with $m$,
	$i=1 \dots k$, such a measure $\lambda$ is called  {\it multistochastic} with respect to given subalgebras
	(bistochastic for  $k=2$); if restrictions are just equivalent to $m$ for $i=1, \dots,k$,
	we call $\lambda$  {\it almost multistochastic}. Finally, if
  $\lambda(U)\leq m(U)$ for any $U\in {\frak A}_i,i=1,\dots,k$, we call $\lambda$
	{\it submultistochastic}.
Of course, bistochastic measure on
$X\times Y$ may be singular with respect to the product measure. For instance, in the case of direct product
of segments $(X,\mu)=(Y,\nu)=[0,1]$ there is a bistochastic measure $\lambda$ on diagonal $\{x=y\}$ (with density $d\mu(x)$).

 Furthermore we suppose for simplicity that $k=2$, i.e. consider functions of two variables.
But there is no serious difference for $k>2$. We consider not only independent variables, most of the notions
may be defined for general pair of sigma-algebras. But even the case of independent variables is often useful
to treat as a general case.

Bistochastic measure on the direct product of spaces define the so called
 \emph{polymorphism} of the space $(X,\mu)$ into  $(Y,\nu)$ (see \cite{V5}),
i.e. ``multivalued mapping'' with invariant measure.
  The case of identified variables $(X,{\frak A},\mu) = (Y,{\frak B},\nu)$ is of special interest: polymorphism
	in this case 	generalizes the concept of automorphism of measure space.
	Almost bistochastic measures define a  {\it polymorphism with quasi invariant measure}.
	Bistochastic or almost bistochastic measure $\lambda$ defines also a bilinear
  (in general case $k$-linear) form $(f(x),g(y))\rightarrow \int f(x)g(y)d\lambda(x,y)$,
	corresponding to the so called Markovian, resp. quasi Markovian operator in corresponding functional
	spaces. Note that this operators $U_{\lambda}$ is a contraction, i.e. has norm at most 1, which preserves the cone of non-negative
	functions. In the case of bistochastic measure this operator (as well as adjoint operator) preserve constants: $U_{\lambda}1=1$.
	
  See \cite{V5,V6,V7} about many connections of polymorphisms  (Markovian operators, joinings, couplings, correspondences, Young measures, bibundles etc).
	Bistochastic measures play a key role in the intensively developing theory of continuous graphs \cite{L}.

Note that for quasi-bistochastic measure$\lambda$ on $X\times Y$ all sets of zero
proper thickness are measurable and have measure 0. Hence all
virtually open sets are $\lambda$-measurable, and therefore
properly virtually continuous functions are
measurable. Equivalent (w.r.t. measure $\mu\times \nu$) properly
virtually continuous functions are also w.r.t $\lambda$ due to Proposition
 \ref{proper-coinciding}. Thus
\emph{for any equivalence class of measurable virtually continuous functions
there is the uniquely well-defined class of $\lambda$-equivalent
$\lambda$-measurable functions.}
Further we formulate this observation as an embedding theorem
of normed spaces.

\subsection{Norm on virtually continuous functions}

Convergence in $\tau$-metric defined above generalizes
convergence in measure for virtually continuous functions. There
are analogues  of known Banach spaces of measurable functions.

A measurable function $h(\cdot,\cdot)$ on the space $(X\times Y,\mu\times \nu)$ is called subbistochastic, if the measure with $\mu\times \nu$-density
 $|h(\cdot,\cdot)|$ is subbistochastic. Denote by  $\mathcal{S}$ the set of subbistochastic functions.

Call a function $f(x,y)=a(x)+b(y)$ {\it separate}.
The following construction
defines a norm (so called {\it regulator norm}) of a function
of two variables, where regulator is separate function and norm
is taken in $L^1$.
Define a finite or infinite norm of a measurable function $f(\cdot,\cdot)$ as

\begin{align*}
\thn{f}:=\inf\Big\{\int_X a(x) d\mu(x)+\int_Y b(y) d\nu(y)\colon &\\
a(x)\geq 0, b(y) \geq 0, |f(x,y)|&\leqmod a(x)+b(y)\Big\}.
\end{align*}

Connection between $\rm{SR}^1$-norm and $\tau$-metric is established in the following
\begin{lemma}
For any function $f$ we have
$\tau(0,f)\leq \sqrt{2\thn{f}}.$
\end{lemma}
\begin{proof}
If $\thn{f}=\infty$, the claim is clear. Assume that $\thn{f}<t^2/2$ for
some $t>0$. Then there exist non-negative functions
$a\colon X \to \mathbb{R}$ and $b\colon Y \to \mathbb{R}$ such that $|f(x,y)|\leqmod a(x)+b(y)$, and $\int_Xa+\int_Yb <t^2/2$. Then Chebyshev inequality
implies
$$
\mu\{x\colon a(x)\geq t/2\}+\nu\{y\colon b(y)\geq t/2\}<t.
$$
But
$$
\{(x,y)\colon f(x,y)\geq t\} \subsetmod \{x\colon a(x)\geq t/2\}\times Y \cup X\times\{y\colon b(y)\geq t/2\}, 
$$
hence $\tau(0,f)<t$, as desired.\end{proof}
\begin{corollary}\label{normtau}
Convergence in $\rm{SR}^1$-norm implies
convergence in $\tau$-metric.
\end{corollary}

Next theorem is an analogue of known L. V. Kantorovich's duality theorem \cite{K} in the mass transportation problem
(concretely, of duality between measures space with Kantorovich distance and and the space of Lipschitz functions, see
also \cite{VK}).

\begin{theorem}\label{dual}
\begin{equation}\label{duall}
\thn{f}=
\sup \left\{\int\limits_{X\times Y} |f(x,y)| h(x,y) d\mu(x) d\nu(y)\colon h\in \mathcal{S}\right\}.
\end{equation}
\end{theorem}
\begin{proof}
Show at first that LHS of \eqref{duall}
is not less than RHS. Indeed, if $h \in \mathcal{S}$, and functions
$a\colon X \to \mathbb{R}$ and $b \colon Y \to \mathbb{R}$ satisfy
$|f(x,y)|\leqmod a(x) + b(y)$, then
\begin{align*}
\int\limits_{X\times Y}|f(x,y) h(x,y)| d\mu(x)d\nu(y) &\leq
\int\limits_{X\times Y} |h(x,y)|(a(x)+b(y)) d\mu(x)d\nu(y)\leq\\
\leq \int\limits_{X} a(x) \int\limits_Y |h(x,y)| d\nu(y) d\mu(x)&+
\int\limits_{Y} b(y) \int\limits_X |h(x,y)| d\mu(x)d\nu(y) \leq\\
&\leq\int\limits_X a(x)d\mu(x)+\int\limits_Y b(y)d\nu(y).
\end{align*}
Taking infimum over admissible pairs of functions $a$ and $b$
we get desired inequality.

It suffices to verify that RHS of \eqref{duall} is not less than LHS.
Choose any $\theta>0$ and consider two convex subsets in
$L^1(X\times Y, \mu\times \nu)$:
\begin{align*}
A&=\{a(x)+b(y)\colon a ,b \geq 0, \int\limits_X a(x)d\mu(x)+\int\limits_Y b(y)d\mu(y)\leq \thn{f}-\theta\},\\
B &= \{g(x,y)\colon g(x,y)\geqmod |f(x,y)|\}.
\end{align*}
By the definition of norm such two sets are disjoint. Let's check that actually
the distance between them in $L^1$ is positive. If not, there exist
sequences of non-negative functions $a_n$, $b_n$, $g_n$
such that $a_n(x)+b_n(y) \in A$, $g_n\in B$ and $\|a_n(x)+b_n(y)-g_n(x,y)\|_{L^1(X\times Y)} \to 0$. Passing to subsequence if necessary we may get that $a_n(x)+b_n(y)-g_n(x,y) \to 0$ $\mu\times \nu-$almost everywhere. Using Koml\'os theorem \cite{JK} we pass to such
a subsequence that $\frac{1}{N}\sum_{k=1}^N a_k(x) \to a(x)$ $\mu-$almost
everywhere, $\frac{1}{N}\sum_{k=1}^N b_k(y) \to b(y)$ $\nu-$almost everywhere
for some non-negative functions $a\in L^1(X)$, $b \in L^1(Y)$.
Consider the function $g(x,y):=a(x)+b(y)$. Then $\frac{1}{N}\sum_{k=1}^N g_n(x,y) \to g(x,y)$ $\mu\times \nu-$almost everywhere. Clearly this implies $g \in B$. On the other
hand,
$$
\int a+\int b\leq \limsup_N \frac{1}{N}\sum_{k=1}^N \left(\int a_k+ \int b_k\right) \leq \thn{f}-\theta
$$
due to semicontinuity of integral of non-negative functions from below
(w.r.t. almost everywhere convergence). Hence $g\in A$. A contradiction.

Now we use separability theorem of Hahn--Banach. Since $A$ contains 0,
there exists a function $h \in L^{\infty}(X\times Y)$ such that
$\int gh< 1 $ for any $g \in A$ and $\int gh > 1$ for any $g \in B$.
Since  $B$ is a translate of a non-negative cone and $\int gh >1$ for any $g \in B$ we get $h \geq 0$ almost everywhere. For any set $X_1 \subset X$ put $a(x)=\frac{\thn{f}-\theta}{\mu(X_1)}\chi_{_{X_1}}(x)$. Then $a(x)\in A$, hence
$$
\frac{1}{\mu(X_1)}\int_{X_1} \int_Y (\thn{f}-\theta) h(x,y)d\nu(y)d\mu(x)\leq 1.
$$
Exchange the variables and write down similar inequalities. It allows
to conclude that the function $\tilde h=(\thn{f}-\theta) h$ belongs to $\mathcal S$.
Bur $|f| \in B$, hence
$$
\int |f| \tilde h =(\thn{f}-\theta) \int |f| h \geq \thn{f}-\theta.
$$
It suffices to remember that $\theta$ is arbitrary.
\end{proof}

Another theorem about $\rm{SR}^1$-norm:
\begin{theorem}\label{normeq}
For any measurable function $f\colon X\times Y \to \mathbb{R}$ inequalities hold\textup:
\begin{equation}\label{normequiv}
\frac{1}{4}\thn{f} \leq \int\limits_0^{\infty}\thi\{|f|\geq\lambda\}d\lambda \leq 2 \thn{f}.
\end{equation}
\end{theorem}
\begin{proof}
If $|f(x,y)|\leqmod a(x)+b(y)$, then
$$
\{(x,y)\colon |f(x,y)|\geq \lambda\}\subsetmod (\{x\colon a(x)\geq \lambda/2\}\times Y)\cup (X\times \{y\colon b(y)\geq \lambda/2\}),
$$
hence
$$
\thi\{|f|\geq \lambda\} \leq \mu\{a\geq \lambda/2\}+ \nu\{b\geq \lambda/2\}.
$$
Integrating by $\lambda$ we get
$$
\int\limits_0^{\infty}\thi\{|f|\geq\lambda\}d\lambda \leq \int_0^\infty (\mu\{a\geq \lambda/2\}+ \nu\{b\geq \lambda/2\})d\lambda=2\left(\int_X a+\int_Y b\right).
$$
Taking infimum by pairs of functions $a$ and $b$
we get the right inequality in \eqref{normequiv}.

Let's prove the left inequality.
Since $\thi\{|f|\geq \lambda\}$ decreases by $\lambda$, we have
\begin{equation}\label{eq101}
\int\limits_0^{\infty}\thi\{|f|\geq\lambda\}d\lambda \geq
\sum_{k\in \mathbb{Z}}2^{k-1}\thi\{|f|\geq 2^k\}.
\end{equation}
For any $\eps>0$ choose sets $A_k$ and $B_k$ so that
$\{|f|\geq 2^k\}\subsetmod (A_k\times Y)\cup (X\times B_k)$ and
$$
\mu(A_k)+\nu(B_k)\leq (1+\eps)\thi\{|f|\geq2^k\}.
$$
Take functions $a(x)=\sum 2^{k+1}\chi_{_{A_k}}(x)$ and
$b(y)=\sum 2^{k+1}\chi_{_{B_k}}(y)$.
It is easy to check that that $|f(x,y)| \leqmod a(x)+b(y)$, hence
$$
\thn{f}\leq \int_Xa+\int_Yb\leq 4(1+\eps)\sum_{k\in \mathbb{Z}}2^{k-1}\thi\{|f|\geq 2^k\}.
$$
The last inequality combined with \eqref{eq101} (and arbitrariness of $\eps$)
finishes the proof.
\end{proof}
This theorem has a useful
\begin{corollary}\label{cut}
If $\thn{f}<\infty$\textup, then a function $f$ is approximated in $\rm{SR}^1$-norm
by its cut-offs.
\end{corollary}
\begin{proof}
Let $f_N$ be two-sided cut-off of $f$
on level $N$. Then $\{|f-f_N|\geq\lambda\}\subset \{|f|\geq N+\lambda\}$, hence
$$
\int\limits_0^\infty \thi\{|f-f_N|\geq\lambda\} d \lambda \leq \int \limits_0^\infty \thi\{|f|\geq N+\lambda\}d\lambda =\int \limits_N^\infty \thi\{|f|\geq\lambda\}d\lambda\to 0
$$
for $N \to \infty$. By Theorem \ref{normeq}, $\thn{f-f_N}\to 0$.
\end{proof}

\begin{theorem}
The closure of step functions in $\rm{SR}^1$-norm consists exactly of all virtually continuous functions
having finite norm (in particular, each bounded virtually continuous function belongs
to this closure).
\end{theorem}

\begin{proof}
Corollary \ref{normtau} and Theorem \ref{tau-virtual} imply
that $SR^1$-limit of step functions is virtually continuous.

Now we have to approximate any virtual
continuous function with finite  $\rm{SR}^1$-norm by step
functions.
Assume that $\thn{f}<\infty$ and $\eps>0$. By Corollary \ref{cut}
two-sided cut-off $f_N$ of the function $f$ approximates $f$: $\thn{f-f_N}<\eps$
for large enough $N$. Fix such $N$. Next, the function $f_N$ is
virtually continuous, hence it is $\tau$-limit of step  functions
by Theorem \ref{tau-virtual}. That is,  $\tau(g,f_N)<\eps/N$ for
some step-function $g$. We may suppose that
absolute values of $g$ do not exceed  $N$ (else replace
$g$ to its cut-off). Since $\tau(g,f_N)<\eps/N$, there
exist sets $X_0\subset X$, $Y_0 \subset Y$ such that
$\mu(X_0)<\eps/N$, $\nu(Y_0)<\eps/N$ and $|f_N-g|<\eps/N$
almost everywhere on $(X\setminus X_0)\times (Y\setminus Y_0)$.
But then we have
$$
|f_N(x,y)-g(x,y)|\leqmod \eps/N+2N\chi_{_{X_0}}(x)+2N\chi_{_{Y_0}}(y),
$$
hence $\thn{f_N-g}\leq \eps/N+2N(\mu(X_0)+\nu(Y_0))\leq 5\eps.$ Thus $\thn{f-g}\leq 6\eps$ and Theorem is proved.
\end{proof}

Denote by $VC^1$ the space of all virtually continuous functions with
finite $\rm{SR}^1$-norm.
It is an analogue of the space $L^1$ for virtually continuous functions
and is a pre-dual for the space of polymorphisms with bounded densities of projections.

\begin{theorem}\label{dualspace}
The space dual to $VC^1$ is a space $QB^{\infty}$ of quasibistoshastic signed measures $\eta$ on $X \times Y$ with finite norm
$$\|\eta\|_{\qbs}=\max\left\{\Big|\Big|\frac{\partial P^x_*|\eta|}{\partial \mu}\Big|\Big|_{L^\infty(X,\mu)},\Big|\Big|\frac{\partial P^y_*|\eta|}{\partial \nu}\Big|\Big|_{L^\infty(Y,\nu)}\right\},$$
where $P^x$ and $P^y$ are projections onto $X$ and $Y$ respectively and $|\eta|$ is a full variation of a signed measure $\eta$.
A coupling between $\eta\in QB^{\infty}$ and $f(x,y)\in VC^1$ is defined as
$\int \tilde{f} d\eta,$ where $\tilde{f}$ is a properly virtually continuous function equivalent to $f$\footnote{
Remark in the end of p. \ref{bsm}\textup guarantees that a function $\tilde{f}$ is $\eta$-measurable\textup,
and the value of integral does not depend on choice of $\tilde{f}$ for fixed $f$.}.
\end{theorem}

In order to prove Theorem \ref{dualspace} we need the following

\begin{lemma} \label{FUNCK}
Let $K$ be a metric compact space\textup, $F$ be a continuous linear functional on the space $C(K)$.
Assume that continuous functions $f_1,f_2,\dots$ on $K$ have uniformly bounded norms and their supporters
are disjoint. Then series $\sum F(f_i)$ converges absolutely. If (defined pointwise)
function $f=\sum f_i$ is continuous, then $F(f)=\sum F(f_i)$.
\end{lemma}
\begin{proof} By Riesz theorem our functional $F$ is integrating over signed Borel measure of finite variation.
Absolute convergence of the above series follows from countable additivity and finiteness of variation.
Equality $F(f)=\sum F(f_i)$ follows from Lebesgue theorem on summable majorant.
\end{proof}

\begin{proof}[Proof of Theorem \ref{dualspace}.]
Let $\eta$ be such a signed measure that $\|\eta\|_{\me}<\infty$.
Note that if for a step function $h$ the estimate $|h(x,y)|\leq a(x)+b(y)$
holds $\mu\times \nu$-almost everywhere, that it holds on the product of sets having full measure,
thus $|\eta|$-almost everywhere. It allows to integrate this inequality over measure $|\eta|$, this gives
\begin{align*}
\left|\int h(x,y) d\eta\right|\leq  \int |h| d|\eta|\leq& \int a(x)+b(y) d|\eta|\leq \\
\leq&\|\eta\|_{\me}\left(\int |a(x)| d\mu(x)+\int |b(y)| d\nu(y)\right).
\end{align*}
Taking infimum in $a,b$ such that $|h|\leqmod a(x)+b(y)$ we get $|\int h d\eta|\leq \thn{h}\|\eta\|_{\me}$,
as desired.

Now we need to show that any continuous functional $F$ on $VC^1$
has such a representation. We may suppose that $\|F\|=1$. For a step set (finite union
of rectangles)
$Z\subset X\times Y$ define
\begin{align*}
\eta(Z)&:=F(\chi_{_{Z}}), \\
|\eta|(Z)&:=\sup \sum_{\sqcup Z_i\subset Z} |\eta (Z_i)|,
\end{align*}
where supremum is taken over all sequences of disjoint step sets $Z_1,\dots$ in $Z$.
Obviously, supremum may be taken over finite families, and we may also take rectangular sets $Z_i$.
Above defined functions of sets are finitely additive.

For any finite family of disjoint step sets $Z_i \subset Z$ we have
$\sum |\eta(Z_i)|=F(\sum \pm \chi_{_{Z_i}})$. Moreover,
$|\pm \chi_{_{Z_i}}|\leq \chi_{_Z}$. If $Z=X_1\times Y$, then $\thn{\sum \pm \chi_{_{Z_i}}}\leq |X_1|$, hence $|\eta|(X_1\times Y)\leq |X_1|$.
Analogously
$|\eta(X\times Y_1)|\leq |Y_1|$. Let's check that finitely additive functions $\eta$ and $|\eta|$,
defined on the algebra of step sets, may be extended to the sign measure and measure on the whole
$\sigma$-algebra $\mathfrak{A}\times \mathfrak{B}$ on the space $X\times Y$.

By Kolmogorov--Hahn criterion it suffices to verify that $|\eta|(Z)=\sum |\eta|(Z_i)$
whenever $Z_i$ are disjoint step sets and $Z=\sqcup Z_i$ is a step set too.
Since $|\eta|$ is premeasure, inequality $|\eta|(Z)\geq \sum |\eta|(Z_i)$ is clear. It suffices to prove the opposite inequality.
By definition $|\eta|(Z)$ is the supremum of sums $\sum |\eta(P_k)|$ over all finite
families of disjoint rectangles $P_k$ in $Z$, hence it suffices to prove that
$
\sum_k |\eta(P_k)|\leq  \sum_i |\eta|(Z_i).
$
Since $|\eta|$ is finitely additive it suffices to prove that
$
|\eta(P_k)|\leq  \sum_i |\eta|(Z_i\cap P_k)
$
for each rectangle $P_k$.
Dividing each set $Z_i \cap P_k$ onto finitely many rectangles we reduce it to inequality like
$$
|\eta(Q)|\leq  \sum_i |\eta(Q_i)|,
$$
where rectangle $Q$ is a union of disjoint rectangles $Q_i$.

Considering a series of cut semimetrics we may easily construct admissible semimetrics
$\rho_X$, $\rho_Y$ such that metric spaces $(X,\rho_X)$, $(Y,\rho_Y)$ are precompact and projections of sides of $Q_i$ and $Q$
have positive distance to their complements. In this case all functions $\chi_{_{Q_i}}$ and $\chi_{_Q}$
are uniformly continuous on $(X\times Y, \rho_X\times \rho_Y)$
and therefore may be extended continuously to its completion
(as 1 to the closure of rectangle and as 0 to the closure of its
complement).
Supporters of extended functions are still disjoint. The space of continuous functions on the completion of $X\times Y$
embeds into $VC^1$ with norm at most 1, hence $F$ acts as a continuous functional on it.
Applying Lemma \ref{FUNCK} to the sequence $\chi_{_{Q_i}}$ we get
$$
\eta(Q)= F(\chi_{_Q})=\sum_i F(\chi_{_{Q_i}})=\sum_i \eta(Q_i),
$$
as desired.
\end{proof}

\begin{corollary} For virtually continuous functions from the space $VC^1$ \textup(in particular\textup,
for bounded virtually continuous functions\textup) there exist well defined
integrals not only over sets of positive measure\textup, as for all summable functions\textup, but
over bistochastic \textup(singular\textup) measures like Lebesgue measure on the diagonal $\{x=y\}\subset [0,1]^2$\textup,
or on graphs of measure preserving maps.
So, virtually continuous functions have a ``trace \textup(restriction\textup) on diagonal'' in the sense of trace theorems.
\end{corollary}

As an application we prove a variant of continuous Hall lemma, Borel version of which is given in appendix
\cite{LA} to the book \cite{L}:

\begin{theorem}\label{Hall}
Let $Z\subset X\times Y$ be virtually closed set. then two following conditions are equivalent:

\textup{(i)} There exists a bistochastic measure $\lambda$ such that $\lambda(Z)=1$\textup;

\textup{(ii)} Proper thickness $\sthi(Z)$ of the set  $Z$ equals $1$.
\end{theorem}

\begin{proof} (ii) immediately follows from (i). Let's prove converse implication.
Without loss of generality the set
$Z$ is closed in the metric $d=\rho_X\times \rho_Y$, where $\rho_X$, $\rho_Y$ are admissible metrics on
$X$, $Y$. Then the function $f(x,y):=\exp(-d((x,y),Z)))$ is properly virtually continuous.

We claim that its $VC^1$-norm equals 1. If not, there exist non-negative functions $a(x)$,
$b(y)$ such that $a(x)+b(y)\geq f(x,y)$ almost everywhere and $\int a+\int b=e^{-2\eps}$ for some $\eps>0$.
Consider $\varepsilon$-neighborhood $Z_{\varepsilon}$ of the set $Z$ in metric $d$.
It is virtually open set, hence by Lemma~\ref{virtual-thickness} its proper thickness coincides
with thickness (and therefore equals 1). We have
$$
a(x)+b(y)\geq e^{-\eps} \chi_{Z_{\varepsilon}}
$$
for almost all $x\in X,y\in Y$, thus
$$
\thi\left(Z_{\varepsilon}\right)\leq e^{\eps} \left(\int a+\int b\right)=e^{-\eps}.
$$
A contradiction.

For an element with unit norm in a Banach space there exists a linear functional of norm
1 which attains its norm on this element. By Theorem
\ref{dualspace} this functional corresponds to a subbistochastic functional $\lambda$, $\int fd\lambda=1$. But
$$
\left|\int fd\lambda\right|\leq \int |f| d|\lambda|=\int_0^1 |\lambda|\left(f^{-1}[t,1])\right) dt\leq 1,
$$ and since all inequalities are just equalities we have $\lambda(Z)=\lambda(f^{-1}\{1\})=1$.
\end{proof}

As usual, Hall lemma admits a standard self-improvement:
\begin{corollary}
For virtually closed $Z$ we always have $\max \lambda(Z)=\sthi(Z),$ where maximum is taken
over all bistochastic measures $\lambda$.
\end{corollary}
\begin{proof}
We use a typical trick: add to spaces $X$ and $Y$ spaces $X_0$, $Y_0$  respectively of measure $1-\sthi(Z)$,
normalize measures on $X\sqcup X_0$ and $Y\sqcup Y_0$.
Consider the set $Z\sqcup X_0\times Y_0\sqcup X_0\times Y\sqcup X\times Y_0$ in
$(X\sqcup X_0)\times (Y\sqcup Y_0)$. It has proper thickness 1. Apply Lemma \ref{Hall} for this set and
find a bistochastic measure in our extended product $(X\sqcup X_0)\times (Y\sqcup Y_0)$. Its part in
$X\times Y$ is what we are searching for.
\end{proof}

 \section{Applications: optimal transport, embeddings theorem, traces of nuclear operators, restrictions of metrics}\label{tvlo}

In this section we mention some applications of the concept of virtual continuity.

\subsection{Kantorovich density in the optimal transportation problem}

We connect Theorem \ref{dualspace} on the space of quasipolymorphisms $QB^{\infty}$
which is dual to the space $VC^1$ of virtually continuous functions, and the classical theorem
of L.~V.~Kantorovich on duality in continuous linear transportation problem. Recall it.

Consider a metric space $(X, \rho)$ \footnote{
in the classical work \cite{K} this space is compact,
but we further need only that it is complete separable metric (=Polish) space}
and two probabilistic Borel measures $\mu_1,\mu_2$ on it .
The following infimum is to be found:
$$\inf\{\int_X \int_X \rho(x,y)d\Psi(x,y): \Psi \in QBS_{\mu_1.\mu_2}^{\infty}\},$$
where $QBS_{\mu_1.\mu_2}^{\infty}$ is a set of measures on $X\times X$ with projections (marginal)
equal to $\mu_1,\mu_2$
(another name for such measures $\Psi$ --- polymorphism from $(X,\mu_1)$ into
$(X,\mu_2)$, or transportation plan, or coupling, or joining, or Young measure etc.)

Main facts known by general name \textit{duality theorem} claim the following
(we use our terminology):

1) above infimum is attained on some non-negative element $\Psi_0$
of the set $QB_{\mu_1.\mu_2}^{\infty}$ (and is not attained, in general, on the set
of absolutely continuous measures like $d\Psi=p(x,y)d\mu_1(x) d\mu_2(y)$,
where $p$ is a measurable summable function);

2) this infimum may be considered as a norm of a signed measure $\|\mu_1-\mu_2\|$
in a certain space of signed Borel measures on the space $(X,\rho)$
with finite variation. \footnote{this observation made in \cite{KR}
originates a tradition to call this norm the Kantorovich--Rubinstein norm, and metric
the Kantorovich metric.}

3) there is a dual definition of the norm
$$
\|\mu_1-\mu_2\|=\sup\left\{\int_X u(x)d(\mu_1-\mu_2)(x): u\in Lip_1(\rho)\right\},
$$ where $Lip_1(\rho)$ is a unit ball in Lipschitz functions space with usual Lipschitz norm. Supremum
is also realized on some Lipschitz function $u_0$ and we have $u(x)-u(y)=\rho(x,y)$
$\Psi_0$-almost everywhere.

Main sense of above claims is that a norm of an element in banach space may be calculated using functional from
the dual space, and this reduces the problem to finding a dual space.

Above claim is known as duality theorem (or optimality criterion) in the optimal transportation problem and was
formulated in the pioneering paper \cite{K}. At fact what is used is that Lipschitz space is Banach dual to the Kantorovich--Rubinstein space.
Let us outline that it is a duality theorem for functions of ``one variable'', and we ``cover'' it by a duality
for functions of two variables.



Below we show how to apply Theorem \ref{dualspace}, which a claim on dual space for the space
$VC^1$ of virtually continuous functions, to Kantorovich duality.
It is more convenient to tell about transportation between two different spaces (of course,
this is equivalent to above problem on the transportation in the same space).

So we get yet another proof of duality theorem,
and the main feature is that our scheme includes spaces of metrics and plans, unlike
original approach of Kantorovich. Choice of spaces
$VC^1$ and $QB^{\infty}$ is natural in the sense that smaller spaces are not enough
(see remark above) and admissible metrics are virtually continuous functions.

 Two-level duality theorem in our specific situation leads, in turn, to following
general two-level duality. We hope that it has another applications.
This is why we start with our general statement and later explain
how to apply it to optimal transportation.

\begin{theorem}\label{abs}
Let $X$ be a real vector space ordered by a convex cone $K$, let $Y=X^*$
be a dual space. Denote by $W$ and $Z=W^*$ two other linear real spaces. Let $A:W\rightarrow X$
be a linear operator and  $B=A^*:Y\rightarrow Z$ be a conjugate operator:

$$
\begin{array}{ccccccccccc}
W& \stackrel{A}{\longrightarrow} & X& \qquad &X^*& \stackrel{A^*}{\longrightarrow} & W^*\\
\end{array}
$$

Fix a positive element $\rho\in K\subset X$ and define (finite or infinite) quasinorm on $Z$ as follows:
$$
\|z\|_{\rho}=\inf\{(y,\rho):By=z,y\geq 0\}.
$$
Assume the following condition: the space $X$ is the sum of the cone $K$ and the space $A(W)$.
Then
$$
\|z\|_{\rho}=\|z\|_{\rho}':=\sup\{(z,u):Au\leq \rho\}.
$$
Moreover, if $\|z\|_{\rho}<\infty$, then there exist
non-negative continuous functional $y\in X^*$ such that
$(y,\rho)=\|z\|_{\rho}$.
\end{theorem}

\begin{remark} In the case when $X$ is a Banach space and cone $K$
is closed and generating ($K-K=X$), classical Kakutani theorem says that a non-negative
functional $y$ on $X$ is automatically norm bounded on $X$.
\end{remark}


\begin{proof}
Obviously $\|z\|_{\rho}'\leq \|z\|_{\rho}$. Indeed, for any element $y\in Y$ such that
$y\geq 0$, $By=z$ and any element $u\in W$ such that $Au\leq \rho$ we have
$$
(z,u)=(By,u)=(y,Au)=(y,\rho)-(y,\rho-Au)\leq (y,\rho),
$$
equality holds if $(y,\rho-Au)=0$.

Assume that  $C:=\|z\|_{\rho}'=\sup\{(z,u):Au\leq \rho\}<\|z\|_{\rho}$. Note that finiteness of $C$ implies that
$(z,w)=0$ whenever $Aw=0$ (else consider $u= \lambda w$ for real $\lambda$ of appropriate sign, coupling $(z,u)$
is unbounded.) It means that $z\in B(Y)$, since the image of a dual operator is just annulator of the kernel
of direct operator. Moreover, if $Aw\geq 0$, then $(z,w)\geq 0$, else consider $u=\lambda w$ with negative $\lambda$.
If $\rho=Au_0$ for some $u$, then for any $y$ with $By=z$ we have $(z,u_0)=(By,u_0)=(y,Au_0)=(y,\rho)$, as desired.
Now let $\rho\notin A(W)$.
We have to find a functional $y\in Y=X^*$
such that $(y,Au)=(z,u)$ for all $u\in W$ (it just means that $z=A^*y=By$), $y\geq 0$, $(y,\rho)=C$.
Such $y$ is already defined on a linear hull of the space $A(W)$ and element $\rho$, and it is nonnegative on this
linear hull. Last claim holds for non-negative elements of the form $\rho-Au$
by definition of value $C$, thus we should check it for non-negative elements of the form
$Au-\rho$. Fix $\varepsilon>0$ and find an element $w\in W$ such that $Aw\leq \rho$ and $(y,Aw)=(z,w)\geq C-\varepsilon$.
We have
$$
(y,Au-\rho)=(y,Au)-C\geq (y,Au)-(y,Aw)-\varepsilon=(y,A(u-w))-\varepsilon\geq -\varepsilon,
$$
since $A(u-w)=(Au-\rho)+(\rho-Aw)\geq 0$, and $y$ on $A(W)$ is nonnegative. Since $\varepsilon>0$ was arbitrary we get
$(y,Au-\rho)\geq 0$.

Since $X$ is the sum of the space $A(W)$ and the cone $K$, Riesz theorem on extending of nonnegative functional
allows to extend $y$ to a nonnegative functional on $X$.
\end{proof}

In our situation $X=VC^1(\Omega_1\times \Omega_2)$, $Y\supset QB^{\infty}(\Omega_1\times \Omega_2)$, where $(\Omega_i,\mu_i),i=1,2$ are
standard probabilistic spaces, $W=L^1(\Omega_1)\oplus L^1(\Omega_2)$, $Z=L^{\infty}(\Omega_1)\oplus L^{\infty}(\Omega_2)$,
operator $A$ maps a pair of functions $u=(w_1(t),w_2(t))\in W$ into $w_1(t_1)+w_2(t_2):=(Au)(t_1,t_2)\in X$,
restriction of $B$ on $QB^{\infty}$ maps quasibistochastic sign measure $\eta$ into pair of its projections on $X$, $Y$:

$$
\begin{array}{ccccccc}
L^1(\Omega_1)\times L^1(\Omega_2)& \stackrel{A}{\longrightarrow} & VC^1(\Omega_1,\Omega_2)& \qquad &QB^\infty(\Omega_1,\Omega_2)& \stackrel{A^*}{\longrightarrow} & L^\infty(\Omega_1)\times L^\infty(\Omega_2)\\
(w_1(x),w_2(y))& \stackrel{A}{\longmapsto} & w_1(x)+w_2(y)& \qquad & \mu &\stackrel{A^*}{\longmapsto} &(Pr_{\Omega_1} \mu,Pr_{\Omega_2} \mu)\\
\end{array}
$$

Element $\rho(t_1,t_2)$ is understood as a price of transporting from $t_1\in \Omega_1$ into $t_2\in \Omega_2$,
take element $z\in Z$ equal to a pair of constant functions $(1,1)$
(we do not lose a generality: for other functions just change measures $\mu_1$, $\mu_2$ onto equivalent).
Note that by definition of the space $VC^1$ each function in this space may be represented as a sum
of nonnegative function and a separate function
$w_1(t_1)+w_2(t_2)$. Thus condition of Theorem \ref{abs} holds. Remark to this theorem (cone of nonnegative functions
is clearly closed and generating in $VC^1$) guarantees that this functional is norm bounded, hence it corresponds
to some polymorphism from $QB^{\infty}$. Norm $\|(1,1)\|_{\rho}$ is infimum of plan prices of transportation $\mu_1$ into $\mu_2$
with price function $\rho$. So, Theorem \ref{abs} implies existence of optimal plan.

If we try to replace $VC^1(\Omega_1\times \Omega_2)$ onto space  $L^1(\Omega_1\times \Omega_2)$, then both assumption and
conclusion of Theorem \ref{abs} fail. In this case nonnegative bounded functional corresponds to bounded function
(not just to polymorphism), and optimal plan may easily not exist.

\subsection{Sobolev spaces and trace theorems}

\begin{theorem}\label{sob} Let $\Omega_1,\Omega_2$ be domains of dimensions $d_1,d_2$ respectively, suppose that
$pl>d_2$ or $p=1,l=d_2$. Then functions from the Sobolev space $W_p^l(\Omega_1\times \Omega_2)$
($l$-th generalized derivatives are summable with power $p$) are virtually continuous as functions of two
variables $x\in \Omega_1$, $y\in \Omega_2$. Embedding $W_p^l(\Omega_1\times \Omega_2)$ into $VC^1 (\Omega_1,K)$
is continuous for any compact subset $K$ of the domain $\Omega_2$.
\end{theorem}

\begin{proof}
Using the theorem of embedding of Sobolev space into continuous functions (see, for instance, \cite{Ma,AF}),
we have the following estimate for functions $h(y)\in W_p^l(\Omega_2)$:
$$
\|h\|_{C(K)}\leq c(\Omega_2,K) \|h\|_{W_p^l(\Omega_2)}.
$$
Let $f(x,y)\in W_p^l(\Omega_1\times \Omega_2)$ be a smooth function.
Set $$a(x):=\|f(x,\cdot)\|_{W_p^l(\Omega_2)}.$$
Then by Fubini's theorem $a\in L^1(\Omega_1)$ and
$$
\int |a|\leq c(\Omega_1,\Omega_2) \|f(x,y)\|_{W_p^l(\Omega_1\times \Omega_2)}.
$$
The following estimate holds on $\Omega_1\times K$:
$$
|f(x,y)|\leq \|f(x,\cdot)\|_{C(K)}\leq c(\Omega_2,K) a(x).
$$
Summarizing this we have
\begin{equation}\label{vlozhenie}
\|f\|_{VC^1(\Omega_1,K)}\leq c(\Omega_1,\Omega_2,K) \|f\|_{W_p^l(\Omega_1\times \Omega_2)}.
\end{equation}
Each function in the class
 $W_p^l(\Omega_1\times \Omega_2)$ is a limit of a sequence of smooth functions,
by (\ref{vlozhenie}) it is a limit in $VC^1$ as well.
\end{proof}

So, under conditions of this theorem we may integrate functions over
quasibistochastic measures. It generalizes usual theorems about traces on submanifolds.

\subsection{Nuclear operators in Hilbert space}

It is well known that the space of nuclear operators in the Hilbert space
$L^2$ is a projective tensor product of Hilbert spaces.
Their kernels are measurable functions of two variables, which can hardly be described directly.
the following theorem claims that kernels of nuclear operators are virtually
continuous as functions of two variables.
Note that kernels of Hilbert--Schmidt operators are not in general
virtually continuous.

\begin{theorem}
Let $(X,\mu)$, $(Y,\nu)$ be standard spaces. the space of kernels of nuclear operators from $L^2(X)$ to $L^2(Y)$
(with Schatten--von Neumann norm) embeds continuously into $VC^1$.
\end{theorem}

\begin{proof}
Let $K(x,y)$ be a kernel of finite rank integral operator from $L^2(X)$ to $L^2(Y)$
with nuclear norm $1$. Then there is a finite sum representation
$K(x,y)=\sum s_k a_k(x) b_k(y)$, where $s_k$ are singular values of the operator,
$(a_k)$, $(b_k)$ are orthonormal systems, $\sum |s_k|\leq 1$.
We have almost everywhere
$$
|K(x,y)|\leq \frac12 \sum |s_k|\cdot |a_k^2(x)|+\frac12 \sum |s_k|\cdot |b_k^2(y)|.
$$
RHS has a form $A(x)+B(y)$, and $\int |A(x)|dx+\int |B(y)|dy\leq 1$.
Thus norm of $K(x,y)$ in the space $VC^1$ does not exceed  1.
It remains to note that any nuclear operator may be approximated in nuclear norm
by operators of finite rank, and by above estimate this is approximation in $VC^1$ as well.
\end{proof}

It implies that such kernels may be integrated not only over diagonal when $X=Y$, which is well known,
but by bistochastic measures. But the space $VC^1$ is wider than kernels of nuclear
operators. If we look at  $VC^1$ as to the space of kernels of integral operators,
it is not unitary invariant, on the contrast to Schatten--von Neumann spaces.
indeed, the definition of $VC^1$ essentially uses known
sigma-subalgebras, which do not have necessary invariance. Close question is
considered in \cite{GD}. See more on traces of nuclear operators and virtual continuity
in \cite{VPZAA}.

\subsection{Restrictions of metrics}

The following problem was one of origins of this paper.
Let $(X,\mu)$ be a standard space with continuous measure.
Assume that $\rho$ is an admissible metric and $\xi$ is a measurable partition of $(X,\mu)$ with parts of null measure
(say, $\xi$ is a partition onto level sets of function which is not constant on sets of positive measure).
May we correctly restrict our metric (a s a function of two variables) onto elements of this partition?

It is not immediately clear, since the metric is a priori just a measurable function.
But admissible metric is virtually continuous, and so for our goal it suffices to
define a bistochastic measure, onto which we have to restrict it. Suppose for simplicity that
$X=[0,1]^2$, $\mu$ is a Lebesgue measure, $\xi$ is a partition onto vertical lines.
Then we say about restriction of virtually continuous function defined on $X^2=[0,1]^4$ onto
three-dimensional submanifold
 $\{(x_1,x_2,x_3,x_4):x_1=x_3\}$.
It is easy to see that such a submanifold equipped by a three-dimensional Lebesgue  measure
defines a bistochastic measure on $X\times X$.

\section{Acknowledgements}

We are grateful to L.~Lovasz, who  sent us his recent monograph \cite{L}, in which close questions
are discussed, and to A.~Logunov for paying our attention to the
possibility of dual definition of the thickness.

Translation by F.~Petrov

\end{document}